\documentclass[10pt]{article}
\usepackage{graphicx,color} 
\usepackage{mathtools,amsmath}
 \usepackage{tikz}
\usepackage[utf8]{inputenc}
\usepackage[T1]{fontenc}
\usepackage{amssymb}
\usepackage{graphicx}
\usepackage{hyperref}
\usepackage{amsmath}
\usepackage{amssymb}
\usepackage{amsthm}
\author{Matthieu Calvez}
\title{Hyperbolicity of the complex of parabolic subgroups of Artin-Tits groups of type $B$}

\newtheorem{theorem}{Theorem}

\newtheorem{lemma}{Lemma}

\newtheorem{proposition}{Proposition}
\newtheorem{Question}{Question}

\newtheorem{Proposition-Definition}{Proposition-Definition}

\begin{document}

\maketitle
\begin{abstract}
We show that the complex of parabolic subgroups associated to the Artin-Tits group of type $B$ is hyperbolic. 
\end{abstract}

\section{Introduction  and background}\label{S:Intro}

An \emph{Artin-Tits group} is a group defined by a finite presentation involving a finite set of generators $S$ in which every pair of generators satisfies at most one balanced relation of the form 
$$\Pi(a,b;m_{ab})=\Pi(b,a;m_{ab}),$$
with $m_{a,b}\geqslant 2$ and where for $k\geqslant 2$, $$\Pi(a,b;k)=\begin{cases} (ab)^{\frac{k}{2}}  & \text{if $k$ is even},\\
(ab)^{\frac{k-1}{2}}a & \text{if $k$ is odd.}
\end{cases}$$

This presentation can be encoded by a \emph{Coxeter graph} $\Gamma$: the vertices are in bijection with the set $S$ and two distinct vertices $a,b$ are connected by an edge labeled $k$ if $k>2$ and labeled by $\infty$ if $a,b$ satisfy no relation. 
The Artin-Tits group defined by the Coxeter graph $\Gamma$ will be denoted $A_{\Gamma}$. 

When $\Gamma$ is connected, the group $A_{\Gamma}$ is said to be \emph{irreducible}. The quotient by the normal subgroup generated by the squares of the elements in $S$ is a \emph{Coxeter group}; when this group is finite, the Artin-Tits group is said to be of \emph{spherical type}. Also, the Artin-Tits group $A_{\Gamma}$ is called \emph{dihedral} if $\Gamma$ has two vertices. 

A proper subset $\emptyset \neq X\subsetneq S$ generates a \emph{proper standard parabolic subgroup} of~$A_{\Gamma}$ which is naturally isomorphic to the Artin-Tits group $A_{\Lambda}$ defined by the subgraph $\Lambda$ of $\Gamma$ induced by the vertices in $X$ \cite{VanDerLek}. A subgroup $P$ of $A_{\Gamma}$ is \emph{parabolic} if it is conjugate to a standard parabolic subgroup. 

The flagship example of an Artin-Tits group of spherical type is Artin's braid group on $n+1$ strands --i.e. the Artin-Tits group defined by the graph $A_n$ shown in Figure \ref{Figure}(a). This group can be seen as the Mapping Class Group of an $(n+1)$-times punctured closed disk $\mathbb D_{n+1}$, that is the group of isotopy classes of orientation-preserving automorphisms of $\mathbb D_{n+1}$ which induce the identity on the boundary of $\mathbb D_{n+1}$.

The group $A_{A_n}$ acts on the set of isotopy classes of non-degenerate simple closed curves in $\mathbb D_{n+1}$ (curves without auto-intersection and enclosing at least 2 and at most $n$ punctures) hence on the \emph{curve graph} of $\mathbb D_{n+1}$. A celebrated theorem of Masur and Minsky states that the curve graph of a surface is hyperbolic \cite{MasurMinsky1}; it is actually the typical example of what is now called a \emph{hierarchically hyperbolic space} \cite{Sisto, BHS1,BHS2}. 

A natural and challenging question is whether, like braid groups, any irreducible Artin-Tits group $A_{\Gamma}$ (of spherical type) admits a nice action on a hierarchically hyperbolic space. 
To answer this question, the first step is to define a hyperbolic space on which $A_{\Gamma}$ acts in the same way as the braid group acts on the curve graph of the punctured disk. 

For spherical type $A_{\Gamma}$, the \emph{complex of irreducible parabolic subgroups} $C_{parab}(A_{\Gamma})$ was recently introduced by Cumplido, Gebhardt, Gonz\'alez-Meneses and Wiest \cite[Definition 2.3]{CGGMW} as a candidate for this purpose. The definition was generalized to $FC$-type Artin-Tits groups by Rose Morris-Wright \cite[Definition 4.1]{MorrisWright}. 
Hyperbolicity of $\mathcal C_{parab}(A_{\Gamma})$ is currently an important open problem and excepted the dihedral case \cite[Proposition 4.4 (i)]{CalvezWiest}, only one positive answer is known. This can be stated as follows: 
\begin{theorem}\cite{CGGMW}\label{T:Intro}
The complex of irreducible parabolic subgroups $\mathcal C_{parab}(A_{A_n})$ is \emph{isometric} to the curve graph 
of the $(n+1)$-times punctured disk $\mathbb D_{n+1}$. Therefore $\mathcal C_{parab}(A_{A_n})$ is hyperbolic. 
\end{theorem}

\begin{figure}
\center
\includegraphics{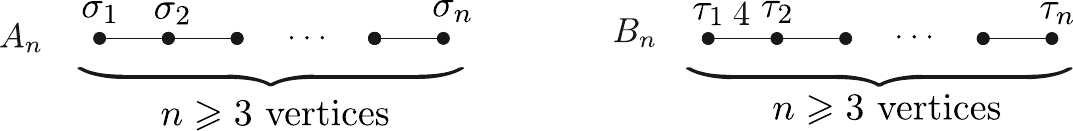}
\caption{The Coxeter graphs of type $A_n$ and $B_n$ with a labeling of the standard generators.}\label{Figure}
\end{figure}

According to the well-known Coxeter classification of finite irreducible Coxeter groups, the list of irreducible Artin-Tits groups of spherical type consists of four infinite families and six ``sporadic" groups; 
in this note we solve the question for one of the infinite families: Artin-Tits groups defined by the graph of type $B_n$ shown in Figure \ref{Figure}(b). We will prove:

\begin{theorem}\label{T:Main}
Let $n\geqslant 3$. The complexes of irreducible parabolic subgroups $\mathcal C_{parab}(A_{B_n})$ and $\mathcal C_{parab}(A_{A_n})$ are quasi-isometric. Therefore $\mathcal C_{parab}(A_{B_n})$ is hyperbolic.
\end{theorem}

For the proof of Theorem \ref{T:Main}, first recall from \cite{CalvezWiest} that any Artin-Tits group $A$ of spherical type contains a(n infinite) generating set $\mathcal{NP}(A)$ 
such that the associated word metric $d_{\mathcal{NP}(A)}$ turns $A$ into a metric space quasi-isometric to $\mathcal C_{parab}(A)$ \cite[Proposition 4.3]{CalvezWiest}: $\mathcal{NP}(A)$ is the set of all elements of $A$ which normalize some proper irreducible standard parabolic subgroup of $A$.

The main ingredient of the proof is a well-known monomorphism $\eta$ from $A_{B_n}$ to $A_{A_n}$; with the notation of Figure \ref{Figure}, 
$\eta$ is defined by $\eta(\tau_1)=\sigma_1^2$ and $\eta(\tau_i)=\sigma_i$ for $i\geqslant 2$. 
The image of $\eta$ is the subgroup $\mathfrak P_1$ of $(n+1)$-strands 1-pure braids, that is the subgroup of all $(n+1)$-strands braids in which the first strand ends in the first position. Note that $\mathfrak P_1$ has index $n+1$ in $A_{A_n}$.  
A presentation for $\mathfrak P_1$ was given by Wei-Liang Chow \cite{Chow} in 1948; for a proof that $\eta$ defines an isomorphism between $A_{B_n}$ and $\mathfrak P_1$, the reader may consult \cite{Peifer}. 

The quasi-isometry promised by Theorem \ref{T:Main} will then be easier to describe using the above-mentionned word-metric model: 
\begin{proposition}\label{Prop:XNPHyp}
Let $n\geqslant 3$. The monomorphism $\eta$ is a quasi-isometry between $(A_{B_n},d_{\mathcal{NP}(A_{B_n})})$ and $(A_{A_n},d_{\mathcal{NP}(A_{A_n})})$.
\end{proposition}

Another interesting infinite generating set of an Artin-Tits group of spherical type was introduced in \cite{CalvezWiest}: $\mathcal P(A)$ is the union of all proper irreducible standard parabolic subgroups together with the cyclic subgroup generated by the square of the Garside element. Similarly to Proposition \ref{Prop:XNPHyp}, we will prove:

\begin{proposition}\label{Prop:XPHyp} Let $n\geqslant 3$. The monomorphism $\eta$ is a quasi-isometry between $(A_{B_n},d_{\mathcal{P}(A_{B_n})})$ and $(A_{A_n},d_{\mathcal{P}(A_{A_n})})$.
\end{proposition}

Recall from \cite{CalvezWiest} that $(A_{A_n},d_{\mathcal P(A_{A_n})})$ is hyperbolic. Actually, $(A_{A_n},d_{\mathcal P(A_{A_n})})$ is quasi-isometric to the complex of arcs in the $(n+1)$-times punctured disk, both of whose extremities are in the boundary of the disk \cite[Proposition 3.3]{CalvezWiest} and this complex is hyperbolic. Proposition \ref{Prop:XPHyp} thus provides a positive partial answer to the conjecture that $(A,d_{\mathcal P(A)})$ is hyperbolic for all $A$ of spherical type. 

Also, Propositions \ref{Prop:XNPHyp} and \ref{Prop:XPHyp} allow to give a partial answer to \cite[Conjecture~4.11(ii)]{CalvezWiest}: 
the identity map $(A_{B_n},d_{\mathcal P(A_{B_n})})$ to $(A_{B_n},d_{\mathcal{NP}(A_{B_n})})$ is Lipschitz --as for any $A$, $\mathcal P(A)\subset \mathcal{NP}(A)$-- but not a quasi-isometry. This follows immediately from Propositions~\ref{Prop:XNPHyp} and \ref{Prop:XPHyp} and the corresponding fact for $A_{A_n}$ \cite[Proposition 4.12]{CalvezWiest}.
  
Finally, there is a third interesting generating set which can be associated to a \emph{spherical} type Artin-Tits group $A$: the set of \emph{absorbable elements} as introduced in \cite{CalvezWiest1,CalvezWiest2}. The corresponding Cayley graph $\mathcal C_{AL}(A)$, called the \emph{additional length graph}, is hyperbolic with infinite diameter for all $A$ of spherical type. 

\begin{Question}
Is it true that $\mathcal C_{AL}(A_{A_n})$ and $\mathcal C_{AL}(A_{B_n})$ are quasi-isometric? 
\end{Question}

\section{Proofs of Propositions 1 and 2}\label{S:Proofs}

{\bf{Notations.}} For simplicity of notation, we will now write $\mathcal P(X_n)$ and $\mathcal {NP}(X_n)$ instead of $\mathcal P(A_{X_n})$ and $\mathcal{NP}(A_{X_n})$, for $X\in \{A,B\}$. 
Associated to the word metric $d_{\mathcal K(X_n)}$ ($\mathcal K\in \{\mathcal P, \mathcal {NP}\}$) on $A_{X_n}$, we have the word-norm $||g||_{\mathcal K(X_n)}=d_{\mathcal K(X_n)}(1,g)$ for $g\in A_{X_n}$. 
The Garside element of $A_{X_n}$ will be denoted by~$\Delta_{X_n}$. The normalizer in $A_{X_n}$ of the (parabolic) subgroup $P$ of $A_{X_n}$ is denoted by~$N_{X_n}(P)$. Finally, for a braid $y\in A_{A_n}$, its image in the symmetric group $\mathfrak S_{n+1}$ will be denoted by $\pi_y$.

\subsection{Braids and curves}\label{S:Curves}
In this section we 
review a geometric perspective on braids and parabolic subgroups of $A_{A_n}$ and establish some useful lemmas.
Recall that the braid group on $(n+1)$-strands can be identified with the Mapping Class Group of a closed disk with $(n+1)$ punctures $\mathbb D_{n+1}$. 

Assume that $\mathbb D_{n+1}$ is the closed disk of radius 
$\frac{n+2}{2}$ centered at $\frac{n+2}{2}$ and the punctures are the integer numbers $1\leqslant i \leqslant n+1$. 
For $1\leqslant i\leqslant n$, the generator~$\sigma_i$ of $A_{A_n}$ corresponds to a clockwise half-Dehn twist along an horizontal arc connecting the punctures $i$ and $i+1$.

The group $A_{A_n}$ acts --on the right-- on the set of isotopy classes of non-degenerate simple closed curves in $\mathbb D_{n+1}$. 
We will abuse notation and write ``curves'' instead of ``isotopy classes of non-degenerate simple closed curves''. 
The action of a braid~$y$ on a curve $\mathcal C$ will be denoted by $\mathcal C^y$ and we think of it as the result of pushing the curve $\mathcal C$ from top to bottom along the braid $y$.

Let $I$ be a proper connected subinterval of $[n]=\{1,\ldots,n\}$, that is $$\emptyset\neq I\subsetneq [n],\ \  \left[(i<j<k) \wedge (i,k\in I)\right] \Longrightarrow j\in I.$$ The proper irreducible standard subgroup of $A_{A_n}$ generated by $\sigma_i, i\in I$, will be denoted by $A_I$. The group $A_{A_n}$ acts --on the right-- on the set of proper irreducible parabolic subgroups by conjugation: for $y\in A_{A_n}$ and $P$ a proper irreducible parabolic subgroup, $P^y=y^{-1}Py$. 

As explained in \cite[Section 2]{CGGMW}, there is a one-to-one correspondence between the proper irreducible parabolic subgroups of $A_{A_n}$ and the curves in $\mathbb D_{n+1}$. Given a proper connected subinterval $I$ of $[n]$, write $m=min(I)$ and $k=\#I$; 
the curve corresponding to the standard parabolic subgroup $A_I$ is the isotopy class of the circle  --or \emph{standard curve}-- $\mathcal C_I$ surrounding the punctures $m,\ldots, m+k$.
More generally, if $P=A_I^y$ is a proper irreducible parabolic subgroup, then the curve corresponding to $P$ is $\mathcal C_I^y$. 
Conversely,  to a curve $\mathcal C$ in $\mathbb D_{n+1}$, we assign the subgroup consisting of all isotopy classes of automorphisms of $\mathbb D_{n+1}$ whose support is enclosed by $\mathcal C$, which is a parabolic subgroup. 

Hence the respective right actions on curves and on parabolic subgroups commute with the above correspondence. 
Notice in particular that the normalizer of a parabolic subgroup $P$ is exactly the stabilizer of the corresponding curve.

For each $1\leqslant i\leqslant n$, define $a_i=\sigma_{i}\ldots \sigma_1$ and let $a_0=Id$. In Figure \ref{Figure:Tau}(i)-(iii) are depicted $a_2, a_8, a_5\in A_{A_9}$, respectively. Observe that for each $0\leqslant i \leqslant n$, $\pi_{a_i}(i+1)=1$. Given $I$, we will be interested in the action of $a_i$ on $\mathcal C_I$. We make the following simple observation.

\begin{lemma}\label{L:ActionOfai}
Let $I$ be a proper connected subinterval of $[n]$, $m=\min(I)$ and $k=\#I$, so that the circle $\mathcal C_I$ surrounds the punctures $m$ to $m+k$. Let $0\leqslant i_0 \leqslant n$. 
\begin{itemize}
\item[(i)] If $i_0+1<m$, i.e. if the puncture $i_0+1$ is to the left of $\mathcal C_I$, then $\mathcal C_I^{a_{i_0}}=\mathcal C_I$.
\item[(ii)] If $i_0+1>m+k$, i.e. if the puncture $i_0+1$ is to the right of $\mathcal C_I$, then $\mathcal C_I^{a_{i_0}}=\mathcal C_{I'}$, where $I'=\{i+1, i\in I\}$.
\item[(iii)] If $i_0+1\in [m,m+k]$, i.e. if the puncture $i_0+1$ is in the interior of $\mathcal C_I$, then $\mathcal C_I^{a_{i_0}}=\mathcal C_I^{a_{m-1}}$.
\end{itemize}
\end{lemma}
\begin{proof}
The contents of Lemma \ref{L:ActionOfai} are depicted in Figure \ref{Figure:Tau}(i) through (iii). Only the third case might need a few words proof. It suffices to observe that (if $i_0+1>m$) the first crossings $\sigma_{i_0},\ldots, \sigma_m$ of $a_{i_0}$ fix the curve $\mathcal C_I$ as they are inner to it. 
\end{proof}

\begin{figure}[hbt]
\center
\includegraphics[scale=0.7]{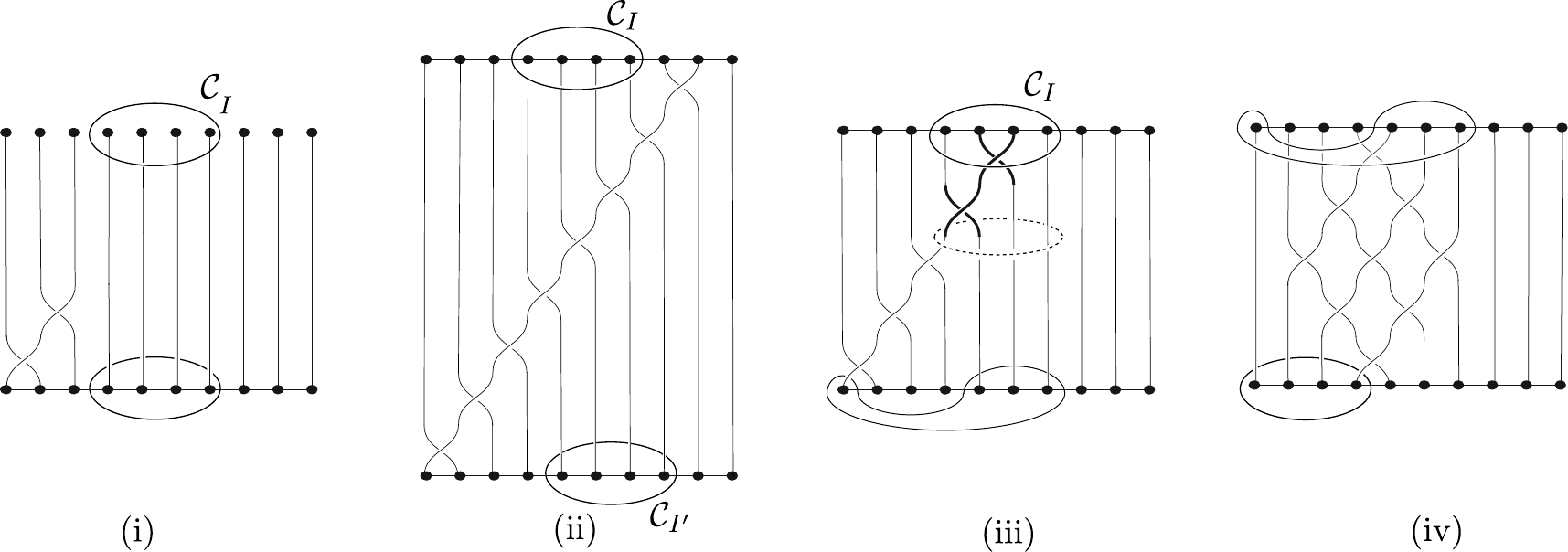}
\caption{In this example, we have $n=9$, $I=\{4,5,6\}$, so that $m=4$, $k=3$.
(i)  The braid $a_2\in A_{A_9}$; this picture illustrates Case (i) in Lemma \ref{L:ActionOfai}: $i_0+1=3<4=m$.
(ii) The braid $a_8\in A_{A_9}$; this picture illustrates Case (ii) in Lemma \ref{L:ActionOfai}: $i_0+1=9>7=m+k$.
(iii) The braid $a_{5}\in A_{A_9}$; this picture illustrates Case (iii) in Lemma \ref{L:ActionOfai}: $i_0+1=6\in [4,7]=[m,m+k]$; 
observe that the action of the bold crossings does not affect the curve $\mathcal C_I$.
(iv) The braid $\tau_I$: its action transforms $\mathcal C_I^{a_{m-1}}$ 
into $\mathcal C_{\{1,\ldots, k\}}$. }
\label{Figure:Tau}
\end{figure}

For each $I$, writing $m=\min(I)$, we define $\mathcal C'_I= \mathcal C_I^{a_{m-1}}$. Observe that $\mathcal C'_I$ is not a standard curve in general (see the bottom part of Figure \ref{Figure:Tau}(iii)); however it can be transformed into a standard curve under the action of a \emph{1-pure braid}. 
To be precise (see an example in Figure \ref{Figure:Tau}(iv)), the 1-pure braid
$$\tau_I=\begin{cases} (\sigma_{m}\ldots \sigma_{m+k-1})\ldots (\sigma_2\ldots \sigma_{k+1}) & \text{if $m>1$},\\ Id & \text{if $m=1$};\end{cases}$$ 
transforms the curve $\mathcal {C'}_I$ into the round curve $\mathcal C_{\{1,\ldots,k\}}$ which surrounds the $k+1$  first punctures. In other words, using the correspondence with proper irreducible parabolic subgroups, we get $A_I^{a_{m-1}\tau_I}=A_{\{1,\ldots,k\}}$. 

%
%

For stating the next lemma, we denote by $sh$ the shift operator: 
\begin{center}
\begin{tabular}{rcl} $A_{\{1,\ldots,n-1\}}$ &  $\longrightarrow$ &   $A_{\{2,\ldots,n\}}$\\ $\sigma_i$ & $\longmapsto$  & $\sigma_{i+1}$.\end{tabular}
\end{center}

\begin{lemma}\label{LemmaTechnical}
Let $I$ be a proper connected subinterval of $[n]$; let $m=\min(I)$ and $k=\#I$. 
Let $0\leqslant i_0,j_0\leqslant n$.  Let $I'=\{i+1,i\in I\}$ (whenever $\max(I)<n$). 
Let $g\in A_I$. Suppose that $z=a_{i_0}^{-1}ga_{j_0}$ is \emph{1-pure}. 
\begin{itemize}
\item[1)] If $i_0+1<m$, then $z=g\in A_I$. 
\item[2)] If $i_0+1>m+k$, then $z=sh(g)\in A_{I'}$ 
\item[3)] If $i_0+1\in [m,m+k]$, then $\tau_I^{-1}z\tau_I\in A_{\{1,\ldots,k\}}$.
\end{itemize}
\end{lemma}

\begin{proof}
First observe that as $z$ is 1-pure, we must have $\pi_g(i_0+1)=j_0+1$. Moreover, as $g\in A_I$, $\pi_g(i)=i$ for all $i\in [1,m-1]\cup[m+k+1,n+1]$. In particular, if $i_0+1<m$ or $i_0+1>m+k$, we must have $i_0+1=j_0+1$, whence $i_0=j_0$.

1) Suppose that $i_0+1<m$; as we have just seen, $z=a_{i_0}^{-1}ga_{i_0}$. But $a_{i_0}$ commutes with all letters $\sigma_i, i\in I$ whence $z=g$. 

2) Suppose that $i_0+1>m+k$; again $z=a_{i_0}^{-1}ga_{i_0}$. We have, for all $i\in I$, 
\begin{eqnarray*}
a_{i_0}^{-1}\sigma_ia_{i_0} & = &( \sigma_1^{-1}\ldots \sigma_{i_0}^{-1})\sigma_i(\sigma_{i_0}\ldots \sigma_1)\\
 & = &\sigma_1^{-1}\ldots\sigma_i^{-1}(\sigma_{i+1}^{-1}\sigma_i\sigma_{i+1})\sigma_i\ldots \sigma_1\\
  & = & \sigma_1^{-1}\ldots\sigma_i^{-1} (\sigma_i \sigma_{i+1}\sigma_i^{-1})\sigma_i\ldots \sigma_1\\
   & = & \sigma_{i+1},
\end{eqnarray*}
and the claim follows.\\
3) Suppose that $i_0+1\in[m,m+k]$; then also $j_0+1\in[m,m+k]$. 
We have 
\begin{eqnarray*}
z & = & a_{i_0}^{-1}ga_{j_0}\\
 & =  & (\sigma_1^{-1}\ldots\sigma_{m-1}^{-1})(\sigma_m^{-1}\ldots \sigma_{i_0}^{-1}g\sigma_{j_0}\ldots\sigma_{m})(\sigma_{m-1}\ldots \sigma_1)\\
  & = & (\sigma_{1}^{-1}\ldots \sigma_{m-1}^{-1}) g' (\sigma_{m-1}\ldots \sigma_1),\\
 \end{eqnarray*}
 where $g'=(\sigma_m^{-1}\ldots \sigma_{i_0}^{-1})g(\sigma_{j_0}\ldots\sigma_{m})\in A_I$ (the first and third factor may be trivial if $i_0+1=m$ or $j_0+1=m$) so that $z\in A_I^{a_{m-1}}$ and then by definition of $\tau_I$, $\tau_I^{-1}z\tau_I\in A_{\{1,\ldots,k\}}$ as claimed. 
\end{proof}

\subsection{The monomorphism $\eta$ and parabolic subgroups}

In this paragraph, we make some useful observations on the monomorphism $$\eta: A_{B_n}\longrightarrow A_{A_n}.$$ 
Recall that $Im(\eta)=\mathfrak P_1$ is the group of $(n+1)$-strands braids whose first strand ends at the first position, that is for $y\in A_{A_n}$, $y\in \mathfrak P_1$ if and only if $\pi_y(1)=1$. In this case, the element $\eta^{-1}(y)\in A_{B_n}$ is well-defined. 


The image of $\eta$ has index $n+1$ in $A_{A_n}$ and two braids $y_1,y_2$ are in the same coset if and only if $\pi_{y_1}(1)=\pi_{y_2}(1)$. 
The braids $a_0,\ldots, a_n$ defined in Section~\ref{S:Curves} provide a finite system of coset representatives of $A_{A_n}$ modulo $\mathfrak P_1$. 
This allows to define a map $\psi:A_{A_n}\longrightarrow A_{B_n}$ in the following way. 
Given $y\in A_{A_n}$, let $i\in \{0,\ldots,n\}$ (it is unique!) be such that $ya_i\in \mathfrak P_1=Im (\eta)$, and define $\psi(y)=\eta^{-1}(ya_i)$.

Given a proper connected subinterval $I$ of $[n]$, we denote by $B_I$ the proper irreducible standard parabolic subgroup of $A_{B_n}$ generated by $\tau_i, i\in I$. The following is immediate; nevertheless we state it as a lemma, for future reference.

\begin{lemma} \label{L:XPAB}
\begin{itemize}

\item [(a)] Let $I$ be a proper connected subinterval of $[n]$;
then ${\eta(B_I)=A_I\cap \mathfrak P_1}$. 
\item[(b)] $\eta(\Delta_{B_n})=\Delta_{A_n}^2$.
\end{itemize}
In particular, $$\eta(\mathcal P(B_n))=(\mathcal P(A_n)\cap \mathfrak P_1)\setminus \{\Delta_{A_n}^{4k+2}, k\in \mathbb Z\}.$$ 
\end{lemma}

\begin{lemma}\label{L:XNPAB}
We have $\eta(\mathcal{NP}(B_n))=\mathcal {NP}(A_n)\cap \mathfrak P_1$. 
\end{lemma}
\begin{proof}
Let $x\in \mathcal{NP}(B_n)$; this means that $x$ normalizes $B_I$ for some $I$. According to \cite[Theorem 6.3]{Paris}, 
$x=uv$, where $u$ commutes with all $\tau_i, i\in I$ and $v\in B_I$. Then 
$\eta(x)=\eta(u)\eta(v)$, where $\eta(u)$ commutes with all $\eta(\tau_i),i\in I$ and $\eta(v)\in A_I$, by Lemma \ref{L:XPAB}(a). As the centralizers of $\sigma_1$ and $\sigma_1^2$ are the same \cite[Theorem 2.2]{FRZ}, $\eta(u)$ commutes with all $\sigma_i,i\in I$ and it follows again from \cite[Theorem 6.3]{Paris} that $\eta(x)$ normalizes~$A_I$; that is $\eta(x)\in \mathcal{NP}(A_n)$. 

Conversely, if $y\in \mathfrak P_1$ normalizes $A_I$, then \cite[Theorem 6.3]{Paris} says that $y=uv$, where $u$ commutes with all $\sigma_i, i\in I$ and $v\in A_I$. 
Suppose that $u,v\notin \mathfrak P_1$; then $1\in I$. As $u$ commutes with $\sigma_1$ and $\pi_u(1)\neq 1$, we must have $\pi_u(1)=2$ and $\pi_u(2)=1$ (observe that $\pi_{\sigma_1}$ is the transposition $[1,2]$) but then 
$u$ cannot commute with $\sigma_2$ whence $I=\{1\}$. We then may write $y=uv=(u\sigma_1)(\sigma_1^{-1}v)$ where $u'=u\sigma_1\in \mathfrak P_1$ commutes with all $\sigma_i, i\in I$ and $v'=\sigma_1^{-1}v\in \mathfrak P_1\cap A_I$. Therefore we may always suppose that $u,v\in \mathfrak P_1$. Taking $a=\eta^{-1}(u)$ and $b=\eta^{-1}(v)$, we get that $a$ commutes with all $\tau_i, i\in I$ and $b\in B_I$ (Lemma~\ref{L:XPAB}(a)). Thus $ab$ normalizes $B_I$ (by \cite[Theorem 6.3]{Paris}), that is, $ab\in \mathcal{NP}(B_n)$ and $y=\eta(ab)\in \eta(\mathcal{NP}(B_n))$.
\end{proof}

\begin{lemma}\label{L:Length}
For all $0\leqslant i\leqslant n$, for all proper connected subinterval $I$ of $[n]$, we have 
\begin{itemize}
\item[(i)] $||a_i||_{\mathcal P(A_n)}\leqslant 2$,
\item[(ii)] $||a_i||_{\mathcal{NP}(A_n)}\leqslant 2$.
\item[(iii)] $||\eta^{-1}(\tau_I)||_{\mathcal{P}(B_n)}=||\eta^{-1}(\tau_I)||_{\mathcal{NP}(B_n)}\leqslant1$. 
\end{itemize}
\end{lemma}
\begin{proof}
If $i<n$, then $a_i\in A_{\{1,\ldots, i\}}\subset N_{A_n}(A_{\{1,\ldots, i\}})$ and ${||a_i||_{\mathcal P(A_n)}=||a_i||_{\mathcal{NP}(A_n)}=1}$; if $i=n$, we can write $a_n=\sigma_n a_{n-1}$ and claims (i), (ii) follow. 
For (iii), it is enough to observe that $\eta^{-1}(\tau_I)$ belongs to $B_{\{2,\ldots,n\}}$. 
\end{proof}

\subsection{Proof of Proposition \ref{Prop:XPHyp}}

The two next statements will immediately imply Proposition \ref{Prop:XPHyp}.

\begin{proposition}\label{Prop:EtaLipschitz}
The homomorphism $\eta$ defines a 1-Lipschitz map 
$$(A_{B_n},d_{\mathcal P(B_n)})\longrightarrow (A_{A_n},d_{\mathcal P(A_n)})$$ for which the map $\psi$ is a quasi-inverse. 
\end{proposition}

\begin{proof}
The first part of the statement follows immediately from Lemma \ref{L:XPAB}, as~$\eta$ is a homomorphism. 
Let us show the second part of the statement.
By construction, for $x\in A_{B_n}$, $\psi\circ\eta(x)=x$. For $y\in A_{A_n}$, $y$ and $\eta\circ\psi(y)$ differ
by $a_i$ for some $0\leqslant i\leqslant n$; Lemma \ref{L:Length}(i) then says that
$Id_{A_{A_n}}$ and $\eta\circ\psi$ are at distance at most $2$ apart.
\end{proof}

\begin{proposition}\label{Prop:PsiLipschitz}
The map $\psi: (A_{A_n},d_{\mathcal P(A_n)})\longrightarrow (A_{B_n},d_{\mathcal P(B_n)})$ is Lipschitz.
\end{proposition}

\begin{proof}
Let $y,y'\in A_{A_n}$ with $d_{\mathcal P(A_n)}(y,y')=1$ and write $g=y^{-1}y'$. This means that $g\in \mathcal P(A_n)$, that is $g=\Delta_{A_n}^{2k}$ for some $k\in \mathbb Z$ or $g\in A_I$ for some $I$.  
There are unique $i_0,j_0\in \{0,\ldots, n\}$ such that $ya_{i_0}\in \mathfrak P_1$ and $y'a_{j_0}\in\mathfrak P_1$. Then also $a_{i_0}^{-1}ga_{j_0}$ is 1-pure. 
Let $x=\eta^{-1}(ya_{i_0})$, $x'=\eta^{-1}(y'a_{j_0})$; then $\psi(y)=x$ and $\psi(y')=x'$. 
We need to estimate $d_{\mathcal P(B_n)}(x,x')$. 

If $g=\Delta_{A_n}^{2k}$, $g$ is central and pure, and as $a_{i_0}^{-1}ga_{j_0}$ is 1-pure, we see that $j_0+1=\pi_g(i_0+1)=i_0+1$ whence $i_0=j_0$ and $a_{ i_0}^{-1}ga_{j_0}=\Delta_{A_n}^{2k}$; that is to say $\eta(x)^{-1}\eta(x')=\Delta_{A_n}^{2k}$, which implies $x^{-1}x'=\Delta_{B_n}^k$ and finally $d_{\mathcal P(B_n)}(x,x')\leqslant 1+ ||\Delta_{B_n}||_{\mathcal P(B_n)}$ (and this bound is just 1 if $k$ is even). 

If $g\in A_I$, Lemma \ref{LemmaTechnical} says that $a_{i_0}^{-1}g a_{j_0}$ or $\tau_I^{-1}(a_{i_0}^{-1}g a_{j_0})\tau_I$ belongs to a standard parabolic subgroup of $A_{A_n}$.
Recall that $\tau_I\in\mathfrak P_1$; the first statement of Lemma \ref{L:XPAB} now says that $x^{-1}x'$ or $\eta^{-1}(\tau_I^{-1})(x^{-1}x')\eta^{-1}(\tau_I)$  belongs to a standard parabolic subgroup of $A_{B_n}$. It follows that 
$d_{\mathcal P(B_n)}(x,x')\leqslant 1+2||\eta^{-1}(\tau_I)||_{\mathcal P(B_n)}\leqslant 3$ (by Lemma \ref{L:Length}(iii)).  

This shows that $d_{\mathcal P(A_n)}(y,y')=1$ implies $d_{\mathcal P(B_n)}(\psi(y),\psi(y'))\leqslant 1+ \max(2,||\Delta_{B_n}||_{\mathcal P(B_n)}).$
\end{proof}


\subsection{Proof of Proposition \ref{Prop:XNPHyp}}

The structure of the proof is very similar to that of Proposition \ref{Prop:XPHyp}: the two next results imply immediately Proposition \ref{Prop:XNPHyp}. The discussion in Section \ref{S:Curves} should be guiding the reader through the proof of Proposition \ref{Prop:PsiNPLipschitz} which is slightly less detailed than its analogue Proposition \ref{Prop:PsiLipschitz}.

\begin{proposition}\label{Prop:EtaNPLipschitz}
The homomorphism $\eta$ defines a 1-Lipschitz map 
$$(A_{B_n},d_{\mathcal{NP}(B_n)})\longrightarrow (A_{A_n},d_{\mathcal{NP}(A_n)})$$ 
for which the map $\psi$ is a quasi-inverse.
\end{proposition}
\begin{proof}
The proof is exactly the same as that of Proposition \ref{Prop:EtaLipschitz}, using now Lemmas \ref{L:XNPAB} and~\ref{L:Length}(ii). 
\end{proof}

\begin{proposition}\label{Prop:PsiNPLipschitz}
The map $\psi: (A_{A_n},d_{\mathcal{NP}(A_n)})\longrightarrow (A_{B_n},d_{\mathcal{NP}(B_n)})$ is Lipschitz.
\end{proposition}

\begin{proof}
Again, we shall bound universally from above the distance $d_{\mathcal{NP}(B_n)}(\psi(y),\psi(y'))$ whenever $d_{\mathcal{NP}(A_n)}(y,y')=1$. 

Let $y,y'\in A_{A_n}$ with $d_{\mathcal{NP}(A_n)}(y,y')=1$ and write $g=y^{-1}y'$. This means that $g\in N_{A_n}(A_I)$ for some $I$. There are unique $i_0,j_0\in\{0,\ldots,n\}$ such that $ya_{i_0}\in\mathfrak P_1$ and $y'a_{j_0}\in\mathfrak P_1$. Then also $a_{i_0}^{-1}ga_{j_0}\in \mathfrak P_1$. 

Let $m=\min(I)$ and $k=\#I$. We have $\pi_g(i_0+1)=j_0+1$ and either both $i_0+1,j_0+1\in [m,m+k]$ or $i_0+1,j_0+1\in [1,m-1]\cup[m+k+1,n+1]$ --this is to say that in $g$ (which stabilizes the curve $\mathcal C_I$), the strand starting at position $i_0+1$ is either inner or outer to the curve $\mathcal C_I$.

Suppose that $i_0+1,j_0+1\in [m,m+k]$; then by construction, $\tau_I^{-1}a_{i_0}^{-1}ga_{j_0}\tau_I$ normalizes $A_{\{1,\ldots,k\}}$. Lemmas \ref{L:XNPAB} and \ref{L:Length}(iii) imply that $d_{\mathcal{NP}(B_n)}(x,x')\leqslant 1+2||\eta^{-1}(\tau_I)||_{\mathcal{NP}(B_n)}\leqslant 3.$

If $j_0+1<m$, then $a_{i_0}^{-1}ga_{j_0}\in N_{A_n}(A_I)$ whence $d_{\mathcal{NP}(B_n)}(x,x')=1$ by Lemma~\ref{L:XNPAB}. 
If $i_0+1,j_0+1>m+k$, then $a_{i_0}^{-1}ga_{j_0}\in N_{A_n}(A_{I'})$ with $I'=\{i+1,i\in I\}$, whence $d_{\mathcal{NP}(B_n)}(x,x')=1$ by Lemma \ref{L:XNPAB}.
If $i_0+1<m$ and $j_0+1>m+k$, then $a_{i_0}^{-1}ga_{j_0}$ conjugates $A_I$ to $A_{I'}$, where $I'=\{i+1,i\in I\}$; in this case, $a_{i_0}^{-1}ga_{j_0}(\sigma_m\ldots\sigma_{m+k})$ normalizes $A_I$ and as $\sigma_m\ldots\sigma_{m+k}\in \mathfrak P_1$, Lemma \ref{L:XNPAB} implies that $d_{\mathcal{NP}(B_n)}(x,x')\leqslant 1+||\eta^{-1}(\sigma_m\ldots\sigma_{m+k})||_{\mathcal{NP}(B_n)}=2.$

It follows that $d_{\mathcal{NP}(B_n)}(\psi(y),\psi(y'))\leqslant 3$, whenever $d_{\mathcal{NP}(A_n)}(y,y')=1$. 
\end{proof}

{\bf{Acknowledgements.}} The author thanks Juan Gonz\'alez-Meneses for pointing a mistake in a preliminary version of this paper. Support by FONDECYT Regular 1180335, MTM2016-76453-C2-1-P and FEDER is gratefully acknowledged.

\end{document}